\documentclass[a4paper,oneside,reqno,11pt]{amsart}
\usepackage[latin1]{inputenc}
\usepackage[english]{babel}
\usepackage[margin=3.75cm]{geometry}
\usepackage{enumerate}
\usepackage{hyperref}
\usepackage{amsmath,amssymb,amsthm}
\newtheorem{theorem}{Theorem}[section]
\newtheorem{corollary}{Corollary}[section]
\newtheorem{lemma}{Lemma}[section]
\newtheorem{remark}{Remark}[section]
\newcommand{\lra}[1]{\langle #1 \rangle}
\newcommand{\rr}{\mathbf{R}}
\newcommand{\ttt}{{\mathbf{T}^n}}
\newcommand{\eff}{{\mathrm{eff}}}
\newcommand{\ve}{\varepsilon}
\newcommand{\mop}[1]{{\mathop{\mathrm{#1}}}}
\newcommand{\inv}{^{-1}}

\title{Subcritical perturbation of a locally periodic elliptic operator}
\author{Klas Pettersson}
\address{University of Tromsø, Norway.}
\email{klas.pettersson@uit.no}
\date{March 3, 2015. Revised May 12, 2016.}

\begin{document}

\begin{abstract}
We consider a singularly perturbed Dirichlet spectral problem for an elliptic operator of second order.
The coefficients of the operator are assumed to be locally periodic and oscillating in the scale $\ve$.
We describe the leading terms of the asymptotics of the eigenvalues and the eigenfunctions to the 
problem, as the parameter $\ve$ tends to zero, under structural assumptions on the potential.
More precisely, we assume that the local average of the potential has a unique global minimum point
in the interior of the domain and its Hessian is non-degenerate at this point.
\end{abstract}

\maketitle

\section{Introduction}

Consider the eigenvalues $\lambda$ and the eigenfunctions $u$ to the following 
Dirichlet spectral problem in a bounded Lipschitz domain\footnote{By a domain we mean a nonempty connected open set in the Euclidean space $\rr^n$.} $\Omega$ in $\rr^n$ with locally periodic coefficients:
\begin{equation}\label{eq:orig}
\begin{aligned}
- \mop{div}\big(  a\big(x,\frac{x}{\ve}\big) \nabla u \big) + \frac{1}{\ve^{\alpha}}c\big(x,\frac{x}{\ve}\big)u & = \lambda u, \quad x \in \Omega, \\
u & = 0, \,\,\,\,\quad x \in \partial \Omega.
\end{aligned}
\end{equation}
We are interested in the asymptotics of the eigenpairs $(\lambda^\ve, u^\ve)$ to \eqref{eq:orig} as the positive parameter
$\ve$ tends to zero.
The value of the parameter $\alpha$ is decisive for the asymptotic behavior of the spectrum. We assume that $0 < \alpha < 2$.

In the present paper we will construct the leading terms in the asymptotics of the eigenvalues $\lambda^\ve$ and an $L^2(\Omega)$ approximation to the eigenfunctions $u^\ve$ to~\eqref{eq:orig}.

We make precise assumptions on the coefficients below.
The differential operator will be assumed to be self-adjoint with real sufficiently smooth coefficients.
The spectrum is then real and discrete.
Also significant for the spectral asymptotics, the potential $c(x,y)$ will be assumed to have a local average with a global unique minimum point in the interior of the domain, and that the local average shows quadratic growth at this point.

The spectral problem \eqref{eq:orig} may be viewed as a perturbation problem for the 
operator $-\mop{div}( a(x,\ve\inv x) \nabla \cdot )$ with $\alpha \in \rr$, for $\ve$ tending to zero.
For $\alpha \le 0$ the perturbation is continuous, while for $\alpha > 0$ it shows singular behavior.
Under the assumption of quadratic growth at the unique minimum point for the corresponding averaged potentials, 
it appears natural to distinguish
four cases of asymptotic behavior for the singular perturbation problem: (i) $0 < \alpha < 4/3$, (ii) $4/3 \le \alpha < 2$, (iii) $\alpha = 2$, and (iv) $\alpha > 2$.

In case (i), the eigenfunctions to \eqref{eq:orig} localize in the scale $\ve^{\alpha/4}$ in the vicinity of a fixed point which is determined by the potential function $c(x,y)$, and the magnitude of the oscillations is $o(1)$ as $\ve$ tends to zero (see for example \cite{PaPe-2014}).
In case (iii), which we call the critical case, the eigenfunctions localize in the scale $\sqrt{\ve}$ at a fixed point which is determined by $a(x,y)$ and $c(x,y)$,
and the magnitude of the oscillations is $O(1)$ as $\ve$ tends to zero, which leads to a shift of order $O(1)$ in the
spectrum as demonstrated in \cite{AlPi-2002}.
Case (iv) has, it seems, not been considered.

Case (ii) is the subject of this paper.
Also here the eigenfunctions localize in the scale $\ve^{\alpha/4}$, as in (i) and (iii).
What makes this case different from both (i) and (iii) is that the magnitude of the oscillations in the
eigenfunctions $u^\ve_j$ are of order $o(1)$ but still big enough to introduce several big terms 
in the asymptotic series for the eigenvalues $\lambda^\ve_j$.
In particular, the bigger $\alpha$ is, the more terms are needed in the series for $\lambda^\ve_j$, and the more correctors are needed to damp the oscillations in $u^\ve_j$.
Moreover, in contrast to both (i) and (iii), the eigenfunctions localize in the vicinity 
of a point $\zeta_\ve$ which moves towards a fixed point but not faster than the rate of localization as $\ve$ tends to zero.
The closer $\alpha$ is to $2$, the slower $\zeta_\ve$ moves.
The position of the point of localization depends in this case on both $a(x,y)$ and $c(x,y)$.
In particular, the main contribution to the energy of the eigenfunctions is given by the second term in the asymptotic series approximation.
The analysis of this case has not been included in the previous studies of the subcritical case~\cite{PaPe-2014, ChPaPe-2014}, and our purpose
here is to fill this gap between (i) and (iii).

The number of terms in the approximations for the eigenvalues $\lambda^\ve_j$ and the
eigenfunctions $u_j^\ve$ given in Theorem \ref{tm:one} below is bounded from
below by $\alpha/(2-\alpha)$ asymptotically as $\alpha$ tends to the critical value $2$ of case (iii).
This is a drawback of our approach since the suggested approximations may be considered unsatisfactory
in a neighborhood of the critical value.

The homogenization of singular perturbed operators where
localization of eigenfunctions occur is well known but still undergoing
exploration.
Examples of scalar second order operators have been given in
\cite{Allaire5, Allaire7, Allaire8, Allaire9, piatnitski2012first, piat2012homogenization, piat2013localization, kreisbeck2014asymptotic, ferreira2014spectral, piatnitski2014ground, PaPe-2014}.
A transport equation has been studied in \cite{Allaire6}, and
scalar fourth order operators have been considered in \cite{AlCaPiSiVa-04, ChPaPe-2014}.
In \cite{AlCaPiSiVa-04} a system of equations have been studied with purely periodic singular perturbation
and locally periodic continuous perturbation.

The rest of this paper is devoted to the description of the leading terms in the spectral asymptotics for the eigenpairs to~\eqref{eq:orig}.

\subsection{Acknowledgment}

I thank the referee for a careful reading of the paper.

\section{Spectral asymtotics}\label{sec:result}

The domain $\Omega$ in \eqref{eq:orig} is assumed to contain the origin (see (A2) below).
The $n$-torus will be denoted by $\ttt$, and it is always realized such that its measure is one.
The average value of a function $v$ on $\ttt$ will be denoted by $\lra{v} = \int_\ttt v(y) dy$.

In expressions involving components of multilinear maps we use the summation convention over repeated indices
when no summation sign $\Sigma$ is written.
We also use the multiindex notation.

We make the following assumptions on the coefficients in~\eqref{eq:orig}:
\begin{enumerate}[($\text{A}$1)]
\item $a(x,y)$ is real, symmetric, belongs to $C^2(\overline{\Omega} \times \ttt)$, and $a(x,y)$ is coercive in $\Omega \times \ttt$: $C>0$, $a_{ij} \xi_{i}\xi_{j} \ge C |\xi|^2$, for all $\xi \in \rr^{n}$.
\item $0 < \alpha < 2$. $c(x,y)$ is real, strictly positive, and belongs to 
$C^{\kappa(\alpha)}(\overline{\Omega}, C^2(\ttt))$, where $q(\alpha)$ is the smallest natural number that is not strictly smaller than any of $2$ and $\alpha/(2 - \alpha)$.
The local average $\lra{c}(x) = \int_\ttt c(x,y) dy$ has a unique global minimum at $x = 0 \in \Omega$, with a non-degenerate Hessian $\nabla \nabla \lra{c}(0)$.
\end{enumerate}

By the Riesz-Schauder and the Hilbert-Schmidt theorems the spectrum of problem~\eqref{eq:orig} is described as follows.
Under the assumptions on the domain $\Omega$, (A1), and (A2), for each $\ve > 0$, the set of eigenvalues to~\eqref{eq:orig} is discrete and may be arranged as
\begin{align*}
0 & < \lambda_1^\ve < \lambda_2^\ve \le \lambda_3^\ve \le \cdots, &
\lim_{j \to \infty} \lambda_j^\ve & = \infty,
\end{align*}
where the eigenvalues are counted as many times as their finite multiplicity.
The eigenfunctions $u^\ve_j$ corresponding to the sequence form a Hilbert basis in $L^2(\Omega)$,
and we assume that they have been orthogonalized.

The spectral asymptotics for~\eqref{eq:orig} will be described using the effective equation
\begin{equation}\label{eq:peff}
\begin{aligned}
- \mop{div}\big( a^{\mathrm{eff}} \nabla p \big)
+ \frac{1}{2}(\nabla \nabla \lra{c}(0)z \cdot z) p
& = \eta p, \quad z \in \rr^n,
\end{aligned}
\end{equation}
where $a^\eff$ is the classical effective matrix corresponding to $a(0,x/\ve)$, defined by
\begin{align*}
a^\mathrm{eff}_{ij}  & = \lra{ a_{ij} }(0) + \lra{ a_{ik} \partial_{k} M_{j} }(0),
\end{align*}
and where $M_{m}(y) : \ttt \to \rr$ are the solutions, normalized by $\lra{M_{m}} = 0$, to 
\begin{align*}
-\mop{div}( a(0,y) \nabla M_m ) & = \partial_{y_i} a_{mi}(0,y), \quad y \in \ttt.
\end{align*}

We will also make use of the functions $N_k(x,y) : \Omega \times \ttt \to \rr$ that are defined as
follows.
Let $\kappa(\alpha)$ be the smallest natural number that is strictly greater than $\frac{\alpha}{2(2-\alpha)}$.
For notation purpose we let $N_0 = 1$.
For $1 \le k \le \kappa(\alpha)$ let $N_k$ be such that
\begin{equation}\label{eq:Nk}
\begin{aligned}
-\mop{div}_y ( a(x,y)\nabla_y N_k ) & =  \lra{ cN_{k-1} }(x) - c(x,y) N_{k-1}(x,y)
 + \sum_{ i = 0 }^{ k - 2 } \epsilon_i N_{ k - 1 - i }(x,y),
\end{aligned}
\end{equation}
with $x \in \overline{\Omega}$, $y \in \ttt$, and normalized by $\lra{N_k} = 0$,
and 
where $\epsilon_i$ are constants that we specify as follows.

Let $\kappa(\alpha)$ defined as in (A2).
Let
\begin{align}\label{eq:Cvezdef}
C_\ve(z) = \sum_{k=0}^{\kappa(\alpha)-1}P(k,\alpha,\ve,z),
\end{align}
where
$P(k,\alpha,\ve,z)$ is the Taylor expansion of $\ve^{-\alpha/2 + k(2-\alpha)}\lra{ cN_k }(\ve^{\alpha/4}z)$ in $z$ of order $q(\alpha)-k$ about $z = 0$.
For fixed $d > 0$, the polynomial function $C_\ve(z)$ has a unique minimum point $\zeta_\ve$ in the ball $|z| \le d\ve^{-\alpha/4 + \delta}$ with $0 < \delta < 2 - \alpha$, for all sufficiently small $\ve$.
This will be proved in Lemma~1 below.
The value of $C_\ve(z)$ at this point admits an expansion of the form
\begin{align}\label{eq:Czeta}
C_\ve(\zeta_\ve) = \sum_{k=0}^{\kappa(\alpha)-1}\ve^{-\alpha/2 + k(2-\alpha)}\epsilon_{k} + o(1),
\end{align}
as $\ve$ tends to zero, 
where $\epsilon_k$  depends only on $N_i$ for $i \le k + 1$, which is a consequence of Lemma~1 (see Section \ref{sec:leadingterms} for a discussion about this).
These are the constants $\epsilon_i$ we choose to use in the definition \eqref{eq:Nk} of $N_k$.

\begin{theorem}\label{tm:one}
Suppose that (A1) and (A2) are satisfied.
Let $\lambda^\ve_j$ and $u^\ve_j$ be the $j$th eigenvalue and eigenvector to~\eqref{eq:orig},
normalized by $\int_\Omega (u^\ve_j)^2 dx = \ve^{n\alpha/4}$.
Let $\zeta_k$ and $\epsilon_k$ be defined by~\eqref{eq:Czeta} and~\eqref{eq:zetam}.
Then, as $\ve$ tends to zero,
\begin{enumerate}[(i)]
\item $\lambda_j^\ve = \sum_{k = 0}^{\kappa(\alpha) - 1} \ve^{-\alpha + k(2 - \alpha)} \epsilon_k + \ve^{-\alpha/2} \eta_j + o(\ve^{-\alpha/2})$,
\item up to a subsequence,
\begin{align*}
u^\ve_j(x) & = p_j\Big( \frac{x - \sum_{k = 1}^{\kappa(\alpha) - 1} \ve^{k(2 - \alpha)} \zeta_k}{\ve^{\alpha/4}} \Big) + o(1) \text{ in } L^2(\Omega),
\end{align*}
\end{enumerate}
where $(\eta_j, p_j)$ is the $j$th eigenpair to \eqref{eq:peff} under a suitable orthogonalization and normalized by $\int_{\rr^n}p_j^2\, dx = 1$.
\end{theorem}

A proof of Theorem~\ref{tm:one} will be presented in the next section.

The leading terms in the series $\epsilon_j$ and $\zeta_j$ can be computed explicitly.
The precise expressions of the terms not presented in the following corollary of Theorem \ref{tm:one} depend on the value of $\alpha$. 
We explain this and give the details on how Corollary \ref{cor:one} is obtained in Section \ref{sec:leadingterms}.
See Remark \ref{rm:one} for a less explicit description of the result.

\begin{corollary}\label{cor:one}
Under the conditions of Theorem \ref{tm:one}, the leading terms in (i) and (ii) of Theorem~\ref{tm:one} are given by
\begin{align*}
\epsilon_0 & = \langle c \rangle (0), & \zeta_1 &  = - \nabla \nabla \langle c \rangle (0)\inv \nabla \langle cN_1 \rangle (0),           \\
\epsilon_1 & = \langle c N_1 \rangle (0),  &  \zeta_2 &  = \langle cN_2 \rangle (0) - \frac{1}{2}  \nabla \nabla \langle c \rangle (0)\inv \langle cN_1 \rangle (0). 
\end{align*}
Moreover, $\epsilon_0 > 0$ and $\epsilon_1 \le 0$ in general,
and $\epsilon_1 < 0$ if the potential $c(x,x/\ve)$ oscillates.
\end{corollary}

By Theorem \ref{tm:one} and Corollary \ref{cor:one} we may explicitly give the leading terms of the asymptotics of the eigenvalues $\lambda_j^\ve$ for $0 < \alpha < 12/7$:
\begin{align*}
\lambda_j^\ve & = \ve^{-\alpha} \langle c \rangle(0) + \ve^{-\alpha/2} \eta_j + o(\ve^{-\alpha/2}), \tag{$0 < \alpha < 4/3$} \\
\lambda_j^\ve & = \ve^{-\alpha} \langle c \rangle(0) + \ve^{-\alpha + 2 - \alpha}\langle cN_1 \rangle(0)  +  \ve^{-\alpha/2} \eta_j + o(\ve^{-\alpha/2}). \tag{$4/3 \le \alpha < 12/7$}
\end{align*}

\begin{remark}\label{rm:one}
The result in Theorem \ref{tm:one} may be stated as follows after suppressing the explicit expansions.
Let $\eta_j^\ve$ and $p_j^\ve$ be defined by 
\begin{align}\label{eq:exact}
\lambda_j^\ve & = \frac{C_\ve(\zeta_\ve)}{\ve^{\alpha/2}} + \frac{\eta_j^\ve}{\ve^{\alpha/2}}, &
u^\ve_j(x) & = p^\ve_j\big( \frac{x}{\ve^{\alpha/4}} - \zeta_\ve \big)r_\ve\big( x, \frac{x}{\ve} \big),
\end{align}
for all sufficiently small $\ve > 0$, with $u_j^\ve$ extended by zero to $\rr^n$ and $r_\ve$ defined by~\eqref{eq:reps}.
Then as $\ve$ tends to zero,
\begin{itemize}
\item $\eta_j^\ve$ converges to $\eta_j$,
\item $p_j^\ve$ converges up to a subsequence to $p_j$ strongly in $L^2(\rr^n)$ and weakly in $H^1(\rr^n)$,
\item the factor $r_\ve(x,x/\ve)$ converges strongly to one in $L^2(\Omega)$ as $\ve$ tends to zero,
\item the point of localization for $u_j^\ve$ in $\Omega$, which is approximated by $\ve^{\alpha/4}\zeta_\ve$,
converges to the minimum point of $\lra{c}(x)$ as $\ve$ tends to zero: $\ve^{\alpha/4}\zeta_\ve = O(\ve^{2 - \alpha})$, and $\zeta_\ve = 0$ for $0 < \alpha < 4/3$,
\end{itemize}
where $\eta_j$ is the $j$th eigenvalue to \eqref{eq:peff} and $p_j$ is a corresponding eigenfunction.

By the regularity assumptions in (A1) and (A2), the functions $N_k$, and therefore also $r_\ve$, are at least in $C^2(\overline{\Omega} \times \mathbf{T}^n)$.
The function that maps $u^\ve_j( \ve^{\alpha/4}(z + \zeta_\ve) )$ to $p_j^\ve(z)$ is bicontinuous on $H^1(\rr^n)$ for all small enough $\ve$.
\end{remark}

\section{Proof of Theorem~\ref{tm:one}}\label{sec:proof}

The plan of the proof of the theorem is to make a suitable change of variables in equation~\eqref{eq:orig}
and proceed to show that in the new variables the Green operator of the corresponding Dirichlet boundary value problem converges uniformly.
Then we obtain the convergence of spectrum from operator theory (see \cite{anselone} and \cite[Lemma 2.6, Theorem 2.2]{AlCo}).
For the passage to the limit involving oscillating functions, we use the two-scale compactness (see \cite[Theorem 4.2]{zhikovtwoscale}, \cite[Theorem 1.2]{Al-1992}, and \cite[Theorem 1]{nguetseng}).
We refer also to~\cite[Chapter III]{oleinik} for an outline of a method of Vishik and Lyusternik~\cite{VishLust} that can be used to prove the convergence of spectrum for certain self-adjoint operators on Hilbert spaces.
Although we do not apply those arguments here, we believe that they would apply and would give the rates of convergence, and avoid the use of two-scale compactness,
because of the fairly high regularity assumptions we have made on the coefficients.
The rate of convergence is desirable is view of the error of $o(\ve^{-\alpha/2})$ in the approximation of the eigenvalues given in (i) in Theorem~\ref{tm:one}.

\begin{lemma}\label{lm:zeta}
Suppose that (A1) and (A2) are satisfied.
Let $d > 0$ and let $\delta$ be such that $0 < \delta < 2 - \alpha$.
Then for all sufficiently small $\ve$ the function $C_\ve(z)$
has a unique minimum point $\zeta_\ve$ in the ball of radius $d \ve^{-\alpha/4 + \delta}$ centered at $0$.
Moreover,
\begin{align}
\zeta_\ve & = \sum_{k=1}^{ \kappa(\alpha) - 1 } \ve^{-\alpha/4 + k(2 - \alpha)} \zeta_k + o( 1 ), \label{eq:zetaexp}
\end{align}
as $\ve$ tends to zero, where $\zeta_m \in \rr^n$ for $1 \le m \le \kappa(\alpha) - 1$ are defined by
\begin{align}\label{eq:zetam}
-\nabla\nabla \lra{c}(0) \zeta_m & = \nabla \lra{ cN_m }(0)
+ \sum_{k=0}^{\kappa(\alpha)-1} \sum_{ \gamma } \sum_{ j } \frac{D^\gamma \lra{cN_k}(0)}{\gamma!} \nabla z^\gamma|_{z = \zeta_j},
\end{align}
and where the sums over $\gamma$ and $j$ are taken over the sets defined by
$\max\{ 3 - k, 2 \} \le |\gamma| \le q(\alpha) - k$
and then $(|\gamma| - 1)j = m - k$.
The coefficients in the asymptotic formula \eqref{eq:Czeta} are for $0 \le m \le \kappa(\alpha) - 1$ given by
\begin{align}\label{eq:epsilonm}
\epsilon_m = \langle c N_m \rangle(0) + \sum_{j = m/2} \frac{1}{2}\nabla \nabla \langle c \rangle (0) \zeta_j \cdot \zeta_j + \sum_{j + k = m} \nabla \langle c N_k \rangle (0) \cdot \zeta_j.
\end{align}
\end{lemma}
\begin{remark}\label{rem:error}
For $0 < \alpha < 4/3$, the exact formulas $\zeta_\ve = 0$ and $C_\ve(\zeta_\ve) = \ve^{-\alpha/2}\epsilon_0 = \ve^{-\alpha/2}\langle c \rangle(0)$ hold.
\end{remark}
\begin{proof}[Proof of Lemma \ref{lm:zeta}]
By the definition \eqref{eq:Cvezdef} of $C_\ve(z)$, after separating the powers of $z$,
\begin{align*}
C_\ve(z)
& = %
\sum_{k=0}^{\kappa(\alpha)-1} \ve^{-\alpha/2 + k(2 - \alpha)} \lra{ cN_k }(0) \\
& \quad + \frac{1}{2}\nabla\nabla\lra{c}(0)z \cdot z + \sum_{k=1}^{\kappa(\alpha)-1} \ve^{-\alpha/4 + k(2 - \alpha)} \nabla\lra{ cN_k }(0) \cdot z \\
& \quad + \sum_{3 \le |\gamma| \le q(\alpha)} \ve^{-\alpha/2 + |\gamma| \alpha/4}\frac{ D^\gamma \lra{ c }(0) }{ \gamma! } z^\gamma \\
& \quad + \sum_{k=1}^{\kappa(\alpha)-1} \sum_{2 \le |\gamma| \le q(\alpha)-k} \ve^{-\alpha/2 + k(2 - \alpha) + |\gamma| \alpha/4}\frac{ D^\gamma \lra{ cN_k }(0) }{ \gamma! } z^\gamma.
\end{align*}
Inside the ball $|z| \le d\ve^{ -\alpha/4 + \delta }$ we have for any $\gamma$ with $2 \le |\gamma| \le q(\alpha)$ that
\begin{align*}
\ve^{-\alpha/2 + |\gamma| \alpha/4} | z^{\gamma} | & \le C \ve^{ (|\gamma| -2)\delta }|z|^2 \\
\ve^{-\alpha/2 + 2 - \alpha + |\gamma| \alpha/4} | z^{\gamma} | & \le C \ve^{ 2 - \alpha + (|\gamma|-2)\delta }|z|^2,
\end{align*}
where the constants depend on $q(\alpha)$, $d$, and $n$.
For all sufficiently small $\ve$, the positive definite quadratic form $\frac{1}{2}\nabla\nabla\lra{c}(0)z \cdot z$ therefore dominates all other terms of order two or greater, for $|\gamma| \ge 3$ and $|\gamma| \ge 2$, respectively, since $0 < \alpha < 2$.
Since $0 < \delta < 2 - \alpha$, the linear term does not perturb the minimum of point of
$\frac{1}{2}\nabla\nabla\lra{c}(0)z \cdot z$ to be located outside the ball $|z| \le d\ve^{-\alpha/4+\delta}$ for any $\ve$ small enough.
Therefore $\zeta_\ve$ is uniquely defined by $\nabla C_\ve(z) = 0$ for all $\ve$ small enough.
The equation $\nabla C_\ve(z) = 0$ also uniquely determines $\zeta_m$ by matching of coefficients, which establishes \eqref{eq:zetaexp} with \eqref{eq:zetam}, and then \eqref{eq:Czeta} with \eqref{eq:epsilonm}.
\end{proof}

Lemma \ref{lm:zeta} justifies the definitions of $N_k$ by \eqref{eq:Nk} under the hypotheses (A1) and (A2).
Let $r_\ve(x,y)$ be defined by
\begin{align}\label{eq:reps}
r_\ve(x,y) & = \sum_{k=0}^{\kappa(\alpha)} \ve^{k(2-\alpha)}N_k(x,y).
\end{align}

\begin{proof}[Proof of Theorem~\ref{tm:one}]
The differential operator in~\eqref{eq:orig},
\begin{align*}
A_\ve u = -\mop{div}\big( a\big( x, \frac{x}{\ve} \big)\nabla u \big)
+ \frac{1}{\ve^\alpha} c\big( x, \frac{x}{\ve}\big) u,
\end{align*}
takes the form
\begin{align*}
\widetilde{A}_\ve p & = -\mop{div}( \hat r_\ve^2 \hat a^\ve \nabla p )
+ \Big( \! -\hat r_\ve \mop{div}( \hat a^\ve \nabla \hat r_\ve ) + \frac{\hat c^\ve}{\ve^{\alpha/2}}\hat r_\ve^2 \Big) p,
\end{align*}
after adding the factor $\hat r_\ve$,
in the variables $z$ and $p$ defined by
\begin{align}\label{eq:newvariables}
z & = \frac{x}{\ve^{\alpha/4}} - \zeta_\ve, &
p(z) \hat r_\ve(z) & = \frac{u(\ve^{\alpha/4}(z + \zeta_\ve))}{\ve^{\alpha/2}},
\end{align}
where $\hat r_\ve(z)$, $\hat a^\ve(z)$, and $\hat c^\ve(z)$ are the functions given by
evaluation of $r_\ve(x,y)$, $a(x,y)$, and $c(x,y)$, respectively,
at $x = \ve^{\alpha/4}(z + \zeta_\ve)$ and $y = \ve^{\alpha/4 - 1}(z + \zeta_\ve)$.
The change of variables~\eqref{eq:newvariables} is well defined for all sufficiently small $\ve > 0$, under the hypotheses (A1) and (A2), using Riesz-Frech\'et representations and Lemma~\ref{lm:zeta}.

The eigenvalues of the self-adjoint elliptic operator $\widetilde{A}_\ve$ with domain $H^1_0(\Omega_\ve)$ need not all be strictly positive, where $\Omega_\ve = \ve^{-\alpha/4}\Omega + \zeta_\ve$.
We therefore make a shift of the spectrum and consider for data $f_\ve$ weakly convergent to $f$ in $L^2(\rr^n)$ the boundary value problem
\begin{equation}\label{eq:bvp}
\begin{aligned}
\widetilde{A}_\ve v - C_\ve(\zeta_\ve) \hat r_\ve^2 v + C \hat r_\ve^2 v
& = f_\ve \hat r_\ve^2, \quad z \in \Omega_\ve,\\
v & = 0, \hspace{.5cm} \quad z \in \partial \Omega_\ve.
\end{aligned}
\end{equation}
Let $G_\ve : L^2(\rr^n) \to L^2(\rr^n)$ be the Green operator that maps $g$ to
the solution $v$ to~\eqref{eq:bvp} with $f_\ve = g$, extended by zero to $\rr^n$.

We claim that $G_\ve$ converges in the $L^2 \to L^2$ norm to 
the Green operator $G_0 : L^2(\rr^n) \to L^2(\rr^n)$ that maps $f$ to the solution $v$ to the boundary value problem
\begin{equation}\label{eq:bvplimit}
\begin{aligned}
-\mop{div}( a^\eff \nabla v )
+ \frac{1}{2}( \nabla\nabla \lra{c}(0)z  \cdot z ) v + C v
& = f , \quad z \in \rr^n.
\end{aligned}
\end{equation}

We first check that there is a positive constant $C$ such that the bilinear form corresponding to the operator in~\eqref{eq:bvp} is coercive on $H^1_0(\Omega_\ve)$ for all sufficiently small $\ve$.
For any positive constant $C$, $-\mop{div} (\hat r^2_\ve \hat a_\ve \nabla \cdot) + C \hat r^2_\ve$ is coercive for all $\ve$ small enough.
It therefore suffices to bound from below the integral 
\begin{align}
\int_{\Omega_\ve} \! \Big( \!  
-\mop{div}( \hat a^\ve \nabla \hat r_\ve )
+ \frac{\hat c^\ve}{\ve^{\alpha/2}}\hat r_\ve 
- C_\ve(\zeta_\ve)\hat r_\ve
  \Big)\hat r_\ve v^2 \,dz, \quad v \in H^1(\rr^n).\label{eq:intrve}
\end{align}

We will consider the integrand in \eqref{eq:intrve} for big and small values of $|z|$ separately.

For big $|z|$ we write
\begin{align*}
& -\mop{div}( \hat a^\ve \nabla \hat r_\ve ) + \frac{\hat c^\ve}{\ve^{\alpha/2}}\hat r_\ve - C_\ve(\zeta_\ve) \hat r_\ve \\
& \quad =
-\mop{div}( \hat a^\ve \nabla \hat r_\ve )
+ \ve^{-2(1- \alpha/4)}[\mop{div}_y( a \nabla_y r_\ve )]\big(  \ve^{\alpha/4}(z + \zeta_\ve), \frac{z + \zeta_\ve}{\ve^{1 - \alpha/4}}  \big) \\
& \qquad + \Big[- \ve^{-2(1- \alpha/4)}\mop{div}_y( a \nabla_y r_\ve ) + \frac{c - \lra{c}}{\ve^{\alpha/2}} r_\ve \Big]\!\big(  \ve^{\alpha/4}(z + \zeta_\ve), \frac{z + \zeta_\ve}{\ve^{1 - \alpha/4}}  \big) \\
& \qquad
+ \frac{\lra{c}( \ve^{\alpha/4}(z+\zeta_\ve) )}{\ve^{\alpha/2}}\hat r_\ve   -   C_\ve(\zeta_\ve) \hat r_\ve.
\end{align*}
Because $0 < \alpha < 2$, and by the choice~\eqref{eq:reps} of $r_\ve$, 
\begin{align*}
 -\mop{div}( \hat a^\ve \nabla \hat r_\ve )
+ \ve^{-2(1- \alpha/4)}[\mop{div}_y( a \nabla_y r_\ve )]\big(  \ve^{\alpha/4}(z + \zeta_\ve), \frac{z + \zeta_\ve}{\ve^{1 - \alpha/4}}  \big) & = o(1),
\end{align*}
and, after using the definition~\eqref{eq:Nk} of $N_k$,
\begin{align*}
\Big[- \ve^{-2(1- \alpha/4)}\mop{div}_y( a \nabla_y r_\ve ) + \frac{c - \lra{c}}{\ve^{\alpha/2}} r_\ve \Big]\!\big(  \ve^{\alpha/4}(z + \zeta_\ve), \frac{z + \zeta_\ve}{\ve^{1 - \alpha/4}}  \big) 
& = O(\ve^{-\alpha/2 + 2 - \alpha}),
\end{align*}
as $\ve$ tends to zero.
By Lemma \ref{lm:zeta} and Remark \ref{rem:error}, $\epsilon_0 = \lra{c}(0)$ and $\ve^{\alpha/4}\zeta_\ve = O(\ve^{2 - \alpha})$ as $\ve$ tends to zero.
It follows that
\begin{align*}
\frac{\lra{c}( \ve^{\alpha/4}(z+\zeta_\ve) )}{\ve^{\alpha/2}}\hat r_\ve   -   C_\ve(\zeta_\ve) \hat r_\ve
& = \frac{\lra{c}( \ve^{\alpha/4}z ) - \lra{c}(0)}{\ve^{\alpha/2}} + O(\ve^{-\alpha/2 + 2 - \alpha}),
\end{align*}
as $\ve$ tends to zero.
By hypothesis (A2),
\begin{align*}
\frac{\lra{c}( \ve^{\alpha/4}z ) - \lra{c}(0)}{\ve^{\alpha/2}} & \ge C |z|^2.
\end{align*}
Therefore, there exist positive constants $C$ and $d$ such that
\begin{align}\label{eq:outside}
\Big( \!
-\mop{div}( \hat a^\ve \nabla \hat r_\ve )
+ \frac{\hat c^\ve}{\ve^{\alpha/2}}\hat r_\ve 
- C_\ve(\zeta_\ve)\hat r_\ve
  \Big)\hat r_\ve
  & \ge C |z|^2,
\end{align}
for all sufficiently small $\ve$ and all $z$ such that $|z| \ge d\ve^{-\alpha/4 + \delta}$ with $0 < \delta \le (2 - \alpha)/2$, because then
$-\alpha/2 + 2 - \alpha \ge -\alpha/2 + 2 \delta$.

We turn to \eqref{eq:intrve} for $|z| < d\ve^{-\alpha/4 + \delta}$.
By the choice of $\hat r_\ve$, making explicit use of the definition of $N_k$, and using the expansion \eqref{eq:Czeta} by Lemma~\ref{lm:zeta} to cancel the terms involving $\epsilon_k$,
\begin{align}
& -\mop{div}( \hat a^\ve \nabla \hat r_\ve ) + \frac{\hat c^\ve}{\ve^{\alpha/2}}\hat r_\ve - C_\ve(\zeta_\ve) \hat r_\ve \notag\\
& \quad =
-\mop{div}( \hat a^\ve \nabla \hat r_\ve )
+ \ve^{-2(1- \alpha/4)}[\mop{div}_y( a \nabla_y r_\ve )]\big(  \ve^{\alpha/4}(z + \zeta_\ve), \frac{z + \zeta_\ve}{\ve^{1 - \alpha/4}}  \big) \notag\\
& \qquad + \Big[- \ve^{-2(1- \alpha/4)}\mop{div}_y( a \nabla_y r_\ve ) + \frac{c - \lra{c}}{\ve^{\alpha/2}} r_\ve \Big]\!\big(  \ve^{\alpha/4}(z + \zeta_\ve), \frac{z + \zeta_\ve}{\ve^{1 - \alpha/4}}  \big)
 \notag\\
& \qquad
+ \frac{\lra{c}( \ve^{\alpha/4}(z+\zeta_\ve) )}{\ve^{\alpha/2}} - C_\ve(z + \zeta_\ve) \notag\\
& \qquad
+ \frac{\lra{c}( \ve^{\alpha/4}(z+\zeta_\ve) )}{\ve^{\alpha/2}}(\hat r_\ve - 1) - C_\ve(\zeta_\ve) (\hat r_\ve - 1) \notag\\
& \qquad + C_\ve(z + \zeta_\ve) - C_\ve(\zeta_\ve) \notag\\
& \quad = C_\ve(z + \zeta_\ve) - C_\ve(\zeta_\ve) + o(1),\label{eq:aux}
\end{align}
as $\ve$ tends to zero, if $\delta$ is such that $\frac{2 - \alpha}{2} \le \delta < 2$ by the choice of $\kappa(\alpha)$.
In the last step we used Taylor's theorem for $0 \le k < \kappa(\alpha)$:
\begin{align*}
& \ve^{-\alpha/2 + k(2-\alpha)}\lra{cN_k}(\ve^{\alpha/4}(z+\zeta_\ve)) \\
& \quad = \sum_{|\gamma| \le q(\alpha) - k} \ve^{-\alpha/2 + k(2 - \alpha) + |\gamma|\alpha/4}
\frac{D^\gamma \lra{cN_k}(0)}{\gamma!} (z + \zeta_\ve)^\gamma \\
& \qquad + \ve^{-\alpha/2 + k(2 - \alpha) + (q(\alpha)-k)\alpha/4}\sum_{|\gamma| = q(\alpha) - k} h_\gamma (\ve^{\alpha/4}(z+\zeta_\ve))(z + \zeta_\ve)^\gamma,
\end{align*}
for some $h_\gamma$ with $\lim_{x \to 0} h_\gamma(x) = 0$, and 
\begin{align*}
\ve^{-\alpha/2 + k(2 - \alpha) + (q(\alpha)-k)\alpha/4}\sum_{|\gamma| = q(\alpha) - k} h_\gamma (\ve^{\alpha/4}(z+\zeta_\ve))(z + \zeta_\ve)^\gamma & = o(1),
\end{align*}
as $\ve$ tends to zero, since $\ve^{\alpha/4}\zeta_\ve = o(1)$ and $|z| \le C\ve^{-\alpha/4 + \delta}$.
Indeed, if $\frac{2 - \alpha}{2} \le \delta < 2$,
\begin{align*}
& -\frac{\alpha}{2} + k(2 - \alpha) + (q(\alpha)-k)\frac{\alpha}{4} + (q(\alpha) - k)( -\frac{\alpha}{4} + \delta ) \ge 0,
\end{align*}
because $q(\alpha)$ has been chosen in (A2) to be such that $-\alpha + q(2-\alpha) \ge 0$.

For $| z | < d \ve^{-\alpha/4 + \delta}$,
the Hessian of $C_\ve(z)$ at $\zeta_\ve$ satisfies
$\nabla\nabla C_\ve(\zeta_\ve) = \nabla \nabla \lra{c}(0) + o(1)$ as $\ve$ tends to zero.
Therefore by (A2) there exists a positive constant $C$ such that
for all sufficiently small $\ve$,
\begin{align}\label{eq:cequi}
C_\ve(z + \zeta_\ve) - C_\ve(\zeta_\ve) & \ge C |z|^2.
\end{align}

By \eqref{eq:outside}--\eqref{eq:cequi}, we see that we may choose $\delta = \frac{2 - \alpha}{2}$ to verify that the estimate in \eqref{eq:outside} is valid for any $z$.
Therefore, there exists a positive constant $C$ such that for any $v \in H^1(\rr^n)$ we have
\begin{align}
 \int_{\Omega_\ve} \!\!\Big( \! 
-\mop{div}( \hat a^\ve \nabla \hat r_\ve )
+ \frac{\hat c^\ve}{\ve^{\alpha/2}}\hat r_\ve 
- C_\ve(\zeta_\ve)\hat r_\ve
  \Big)\hat r_\ve v^2 \,dz
  & \ge C \int_{\Omega_\ve} |z|^2 v^2 \,dz,\label{eq:coef}
\end{align}
for all sufficiently small $\ve$.

It follows from~\eqref{eq:coef} that there exist positive constants $C$ and $C_1$ such that
\begin{align*}
& \int_{\Omega_\ve} \hat r_\ve^2 \hat a^\ve \nabla v \cdot \nabla v \, dz
+ \int_{\Omega_\ve} \Big( \! -\hat r_\ve \mop{div}( \hat a^\ve \nabla \hat r_\ve ) + \frac{\hat c^\ve}{\ve^{\alpha/2}}\hat r_\ve^2 \Big) v^2 \, dz \\
& \qquad + (-C_\ve(\zeta_\ve) + C)\int_{\Omega_\ve}  \hat r_\ve^2 v^2 \,dz \\
& \quad \ge C_1 ( \| \nabla v \|^2_{ L^2(\Omega_\ve) } + \| v \|^2_{ L^2(\Omega_\ve) } + \| |z| v \|^2_{ L^2(\Omega_\ve) } ),
\end{align*}
for all $v \in H^1(\rr^n)$ and all sufficiently small $\ve > 0$.
This inequality gives the coerciveness of the bilinear form corresponding to the differential operator in~\eqref{eq:bvp}.
Moreover, the term weighted by $|z|^2$ gives compactness in $L^2(\rr^n)$.

Let $v_\ve$ be the sequence of solutions to~\eqref{eq:bvp} in $H^1_0(\Omega_\ve)$ that are extended by zero to $\rr^n$.
By the strong convergence of $\hat r_\ve^2$ (to one) and the boundedness of $f_\ve$, we have the a priori estimate
\begin{align*}
\| \nabla v_\ve \|_{L^2(\rr^n)}
+
\| v_\ve \|_{L^2(\rr^n)}
+
\| |z| v_\ve \|_{L^2(\rr^n)}
& \le 
C.
\end{align*}
It follows that there exists a subsequence, we still denote it by $\ve$, such that $v_\ve$ converges weakly to some $v$ in $H^1(\rr^n)$, and by compactness, $v_\ve$ converges strongly to $v$ in $L^2(\rr^n)$.
Moreover, by two-scale compactness, 
$\nabla v_\ve$ converges to $\nabla v + \nabla_\xi M_k(\xi) \partial_k v$ weakly two-scale in $L^2(\rr^n)$ in the scale $1 - \alpha/4$, along a subsequence.

By passing to the limit in the variational form of equation~\eqref{eq:bvp} we see that any convergent subsequence of $v_\ve$ converges strongly in $L^2(\rr^n)$, and weakly in $H^1(\rr^n)$, to the unique solution $v$ to equation~\eqref{eq:bvplimit}.

Using the compactness of $G_\ve$ and $G_0$, and the convergence of the solution to~\eqref{eq:bvp} to the solution to~\eqref{eq:bvplimit}, we conclude that $\| G_\ve - G_0 \|_{L^2 \to L^2}$ goes to zero as $\ve$ tends to zero.

The operators $G_\ve$ and $G_0$ are compact and self-adjoint, and $G_\ve$ converges uniformly
to $G_0$.
It follows that the spectrum of $G_\ve$ converges to the spectrum of $G_0$ in the sense of the
statement of the theorem.
A shift back by $C$ concludes the proof.
\end{proof}

\section{Computing the leading terms}\label{sec:leadingterms}

In this section we show how to obtain explicit expressions for the leading terms in the asymptotic series of Theorem \ref{tm:one}
and in that way prove Corollary \ref{cor:one}.
We also tabulate the ranges for $\alpha$ that correspond to distinct pairs of $(\kappa(\alpha),q(\alpha))$, for small values of $\alpha$.

A careful look at the definitions of $N_k$, $\epsilon_k$, and $\zeta_k$ reveals that  
the terms in the sequences can be computed in the following order:
$
N_1, \epsilon_0, \zeta_1, 
N_2, \epsilon_1, \zeta_2, 
N_3, \epsilon_2, \zeta_3, 
\cdots . 
$
The first few terms are defined by the following equations.
By \eqref{eq:Nk},
\begin{align*}
- \mop{div} ( a\nabla N_1 ) & = \langle c \rangle - c \\
- \mop{div} ( a\nabla N_2) & = \langle cN_1 \rangle - cN_1 + \epsilon_0 N_1 \\
- \mop{div} ( a\nabla N_3 ) & = \langle cN_2 \rangle - cN_2 + \epsilon_0 N_2 + \epsilon_1 N_1 \\
- \mop{div} ( a\nabla N_4 ) & = \langle cN_3 \rangle - cN_3 + \epsilon_0 N_3 + \epsilon_1 N_2 + \epsilon_2 N_1 \\
\vdots
\end{align*}
By \eqref{eq:zetam} and \eqref{eq:epsilonm} in Lemma \ref{lm:zeta},
\begin{align*}
\epsilon_0 & = \langle c \rangle (0) \\
\epsilon_1 & = \langle c N_1 \rangle (0) \\
\epsilon_2 & = \langle c N_2 \rangle (0)  + \frac{1}{2} \nabla \nabla \langle c \rangle (0) \zeta_1 \cdot \zeta_1 + \nabla \langle c N_1 \rangle(0) \cdot \zeta_1 \\
\epsilon_3 & = \langle c N_3 \rangle (0)  + \nabla \langle c N_1 \rangle(0) \cdot \zeta_2 + \nabla \langle c N_2 \rangle(0) \cdot \zeta_1 \\
\vdots
\end{align*}
and
\begin{align*}
- \nabla \nabla \langle c \rangle (0) \zeta_1 & = \nabla \langle cN_1 \rangle (0) \\
- \nabla \nabla \langle c \rangle (0) \zeta_2 & = \nabla \langle cN_2 \rangle (0)  \\
                        & \quad                 + \sum_{3 = |\gamma| \le q(\alpha)} \frac{D^\gamma \langle c \rangle(0)}{\gamma!} \nabla z^\gamma|_{z = \zeta_1} \\
                           & \quad                       + \sum_{2 = |\gamma| \le q(\alpha) - 1} \frac{D^\gamma \langle cN_1 \rangle(0)}{\gamma!} \nabla z^\gamma|_{z = \zeta_1} \\
- \nabla \nabla \langle c \rangle (0) \zeta_3 & = \nabla \langle cN_3 \rangle (0)  \\
                        & \quad                 + \sum_{4 = |\gamma| \le q(\alpha)} \frac{D^\gamma \langle c \rangle(0)}{\gamma!} \nabla z^\gamma|_{z = \zeta_1} \\
                           & \quad                       + \sum_{3 = |\gamma| \le q(\alpha) - 1} \frac{D^\gamma \langle cN_1 \rangle(0)}{\gamma!} \nabla z^\gamma|_{z = \zeta_1} \\
                                                      & \quad                       + \sum_{4 = |\gamma| \le q(\alpha) - 2} \frac{D^\gamma \langle cN_2 \rangle(0)}{\gamma!} \nabla z^\gamma|_{z = \zeta_2} \\
\vdots                            
\end{align*}
The last three equations show that only a few of the terms $\epsilon_k$ and $\zeta_k$ have simple explicit expressions that do not depend on $\alpha$. 
These particular terms are
\begin{align*}
\epsilon_0 & = \langle c \rangle (0), & \zeta_1 &  = - \nabla \nabla \langle c \rangle (0)\inv \nabla \langle cN_1 \rangle (0),           \\
\epsilon_1 & = \langle c N_1 \rangle (0),  &  \zeta_2 &  = \langle cN_2 \rangle (0) - \frac{1}{2}  \nabla \nabla \langle c \rangle (0)\inv \langle cN_1 \rangle (0). 
\end{align*}
Under our hypotheses (A1) and (A2), we have in particular,
\begin{align*}
\epsilon_0 &  > 0, & 
\epsilon_1 & \le 0,
\end{align*}
where the second inequality is strict precisely when $c(x,y)$ is not constant in $y$.
Indeed, $\langle c \rangle (0) > 0$ by assumption.
To see the negativity of the second term one uses $N_1$ as a test function in the variational form of the equation for $N_1$ to obtain
\begin{align*}
\langle c N_1 \rangle(0) & = - \int_{\mathbf{T}^n} a(0,y) \nabla_y N_1(0,y) \cdot \nabla_y N_1(0,y) \, dy < 0,
\end{align*}
because $N_1(0,y)$ is nonzero if $c(x,y)$ is not constant in $y$, using the positive definiteness of $a(0,y)$ twice.
Otherwise $\langle c N_1 \rangle(0) = 0$.

Corollary \ref{cor:one} has now follows from Theorem \ref{tm:one}.\\

In Theorem \ref{tm:one} we have established asymptotic series approximations for the eigenpairs $(\lambda_j^\ve, u_j^\ve)$
to the spectral problem~\eqref{eq:orig}.
The number of terms that are included in the series depend on the value of $\kappa(\alpha)$, and most terms depend on the value of $q(\alpha)$.
The function $\kappa(\alpha)$ is right-continuous and jumps at $\alpha = 4b/(1+2b)$ for $b = 1, 2, \ldots$.
The function $q(\alpha)$ is left-continuous and jumps at $\alpha = 2b/(1+b)$ for $b = 2, 3, \ldots$.
The set of jumps for $\kappa(\alpha)$ is a subset of the set of jumps for $q(\alpha)$.
For distinct pairs $(\kappa(\alpha), q(\alpha))$ we have in general a distinct polynomial functions $C_\ve(z)$ defined by \eqref{eq:Czeta}.
The first few ranges and values of $\alpha$ are presented in Table~\ref{tb:one}, as well as the terms that are present in the corresponding series.

\begin{table}[h!]
\begin{tabular}{rccl}
Range of $\alpha$ & $\kappa(\alpha)$ & $q(\alpha)$ & Terms \\[3mm]
$0 < \alpha < 4/3$  & $1$ & $2$ & $\epsilon_0$  \\
$\alpha = 4/3$ & $2$ & $2$ & $\epsilon_0, \epsilon_1, \zeta_1$ \\
$4/3 < \alpha \le 3/2$  & $2$ & $3$ & $\epsilon_0, \epsilon_1, \zeta_1$ \\
$3/2 < \alpha < 8/5$ & $2$ & $4$ & $\epsilon_0, \epsilon_1, \zeta_1$ \\
$\alpha = 8/5$ & $3$ & $4$ & $\epsilon_0, \epsilon_1, \epsilon_2, \zeta_1, \zeta_2$ \\
$8/5 < \alpha \le 5/3$  & $3$ & $5$ & $\epsilon_0, \epsilon_1, \epsilon_2, \zeta_1, \zeta_2$ \\
$5/3 < \alpha < 12/7$ & $3$ & $6$ & $\epsilon_0, \epsilon_1, \epsilon_2, \zeta_1, \zeta_2$ \\
$\alpha = 12/7$ & $4$ & $6$ & $\epsilon_0, \epsilon_1, \epsilon_2, \epsilon_3, \zeta_1, \zeta_2, \zeta_3$ \\
$12/7 < \alpha \le 7/4$  & $4$ & $7$ & $\epsilon_0, \epsilon_1, \epsilon_2, \epsilon_3, \zeta_1, \zeta_2, \zeta_3$ \\
$7/4 < \alpha < 16/9$ & $4$ & $8$ & $\epsilon_0, \epsilon_1, \epsilon_2, \epsilon_3, \zeta_1, \zeta_2, \zeta_3$ \\[3mm]
\end{tabular}
\caption{The leading terms in the series of pairs $(\kappa(\alpha), q(\alpha))$.\label{tb:one}}
\end{table}

\bibliographystyle{plain}
\bibliography{refs}

\end{document}